\documentclass[12pt]{article}
   
\usepackage{amsmath,enumerate, amsthm}
\usepackage{amsfonts}
\usepackage{amssymb}

\DeclareMathOperator{\ex}{ex}
\newcommand{\ind}{\text{-}\mathrm{ind}}
\newcommand{\exi}[3]{\ex(#1, \{#2, #3\ind\})}

\newtheorem{theorem}{Theorem}[section]

\newtheorem{corollary}[theorem]{Corollary}

\newtheorem{lemma}[theorem]{Lemma}

\begin{document}

\title{Induced Tur\'an numbers}
\author{Po-Shen Loh \thanks{Department of Mathematical Sciences, Carnegie
  Mellon University, \texttt{ploh@cmu.edu}. Research supported in part by
National Science Foundation CAREER Grant DMS-1455125.} 
\and 
Michael Tait\thanks{Department of Mathematical Sciences, Carnegie Mellon
  University, \texttt{mtait@cmu.edu}. Research supported in part by National Science Foundation postdoctoral fellowship DMS-1606350} 
\and 
Craig Timmons\thanks{Department of Mathematics and Statistics, California State University Sacramento, \texttt{craig.timmons@csus.edu}.
Research supported in part by Simons Foundation Grant \#359419.}
\and
Rodrigo M. Zhou\thanks{COPPE, Universidade Federal do Rio de Janeiro, \texttt{rzhou@cos.ufrj.br}.}
}
\date{}

\maketitle

\abstract{The classical K\H{o}v\'ari--S\'os--Tur\'an theorem states that if
$G$ is an $n$-vertex graph with no copy of $K_{s,t}$ as a subgraph, then
the number of edges in $G$ is at most $O(n^{2-1/s})$. We prove that if one
forbids $K_{s,t}$ as an \emph{induced}\/ subgraph, and also forbids
\emph{any}\/ fixed graph $H$ as a (not necessarily induced) subgraph, the
same asymptotic upper bound still holds, with different constant factors.
This introduces a nontrivial angle from which to generalize Tur\'an theory
to induced forbidden subgraphs, which this paper explores. Along the way,
we derive a nontrivial upper bound on the number of cliques of fixed order
in a $K_r$-free graph with no induced copy of $K_{s,t}$.  This result is an
induced analog of a recent theorem of Alon and Shikhelman and is of
independent interest.}


\section{Introduction}

Tur\'an-type problems represent some of the oldest investigations in
Extremal Combinatorics, with many intriguing questions still notoriously
open. They share a common theme of asking for the maximum number of edges
in a graph (or similar combinatorial structure) with a given number of
vertices, subject to the condition of forbidding certain substructures. In this
paper, we open the systematic study of a natural yet new direction in this
area, focusing on induced substructures, and demonstrate connections
between existing areas of research and the new results and problems.

The most basic Tur\'an question concerns ordinary graphs and asks to
determine $\ex(n, H)$, defined as the maximum number of edges in an
$n$-vertex graph with no subgraph isomorphic to $H$. Tur\'an's original
theorem \cite{turan} solves this completely when $H$ is a complete graph.
For non-complete $H$, the condition obviously does not require the
forbidden subgraph to be induced, or else the answer would trivially be
$\binom{n}{2}$.  The classical Erd\H{o}s--Stone--Simonovits theorem \cite{erdos-stone} shows that the asymptotic behavior is
determined by the chromatic number $\chi(H)$, namely 
\begin{equation}\label{erdos stone}
\ex(n, H) = \left(1 - \frac{1}{\chi(H)-1}\right) \binom{n}{2} + o(n^2). 
\end{equation}
This determines $\ex ( n , H)$ asymptotically 
for non-bipartite $H$. For bipartite $H$, it is
often quite difficult to obtain good estimates on the Tur\'an number.
The classical theorem of K\H{o}v\'ari, S\'os, and Tur\'an states that $\ex(n, K_{s,t}) < c_{s,t} n^{2-\frac{1}{s}}$ but this 
is overwhelmed by the $o(n^2)$ error term in (\ref{erdos stone}).  Many interesting
and longstanding open problems remain unsolved in this case, often called the
\emph{degenerate case}, as surveyed by F\"uredi and Simonovits \cite{fs}
and Sidorenko \cite{sidorenko}.

Many other generalizations have been considered, such as to hypergraphs where even the most basic questions remain
unanswered, or to non-complete host graphs, or other
combinatorial objects such as partially ordered sets. In all contexts, analogous questions with
multiple simultaneously forbidden sub-configurations have been
studied.

\subsection{Induced substructures}

Although the opening section dismissed as trivial the situation of induced
subgraphs in the ordinary graph Tur\'an problem, it turns out that this
first impression is wrong, and there are natural and interesting
questions. Induced Tur\'an-type problems have previously
surfaced in many of the above contexts. On the topic of one of the central
open problems in hypergraph Tur\'an theory, Razborov
\cite{razborov-induced} established the conjectured upper bound for
$K_4^{(3)}$-free hypergraphs under the additional condition of forbidding
induced sub-hypergraphs with four vertices and exactly one edge. In the
context of partially ordered sets, the induced Tur\'an problem is
nontrivial because not all sets are comparable, and this has been studied
as well \cite{boehnlein-jiang-induced, lu-milans-induced}. 

It has been less clear what induced question to study in the original graph
context. In the late 1980's, F.\ Chung, Gy\'{a}rf\'{a}s, Trotter, and Tuza 
\cite{chung-gyarfas-tuza-trotter} studied a version which was posed in \cite{bermond-etal} and also by Ne\v{s}et\v{r}il and
Erd\H{o}s, in which the maximum degree was specified instead of the number
of vertices. Specifically, they determined the maximum number of edges in a
connected graph with maximum degree $\Delta$ and no induced subgraph
isomorphic to the 4-vertex graph formed by two vertex-disjoint edges. Several other authors continued
this line of investigation with different forbidden induced subgraphs
\cite{chung-jiang-west, chung-west-2p3, chung-west-p4}. However, this
quantity is usually infinite unless the forbidden induced subgraph has a
very simple structure (generally disjoint unions of paths).

Around that time, while studying hereditary properties, Pr\"omel and Steger
\cite{promel-steger1, promel-steger2, promel-steger3} introduced another
extremal induced subgraph problem: determine the maximum number of edges a
graph $G = (V_n, E)$ can have such that there exists a graph $G_0 = (V_n,
E_0)$ on the same vertex set with $E_0 \cap E = \emptyset$ such that $(V_n,
E_0 \cup X)$ does not contain an induced $H$-subgraph for all $X \subset
E$. This was natural in the context of their investigation of counting the
number of graphs in a hereditary family, and generalized the Erd\H{o}s,
Frankl, and R\"odl \cite{efr-count} estimate on the number of $H$-free
graphs being $2^{(1+o(1))\ex(n, H)}$, to induced-$H$-free graphs. Rates of
growth of hereditary properties were further studied by several researchers
(e.g., Bollob\'as and Thomason \cite{bollobas-thomason}, and Balogh,
Bollob\'as, and Weinreich \cite{balogh-bollobas-weinreich}).

\subsection{New general problem}

When a single non-complete graph $F$ is forbidden as an induced subgraph,
the maximum number of edges is trivially $\binom{n}{2}$. We introduce the
question of simultaneously forbidding both an induced copy of $F$
and a (not necessarily induced) copy of $H$, defining 
\[
\exi{n}{H}{F}
\]
to be the maximum number of edges over all such graphs with $n$ vertices. The
answer is no longer trivially $\binom{n}{2}$ because $H$ is not necessarily
induced, and this general question is related to two areas of Extremal
Combinatorics which have received much attention: Ramsey--Tur\'an Theory and
the Erd\H{o}s--Hajnal Conjecture.

Introduced by S\'os \cite{sos}, the Ramsey--Tur\'an number $\textbf{RT}(n,
H, m)$ is the maximum number of edges that an $n$-vertex graph with
independence number less than $m$ may have without containing $H$ as a (not
necessarily induced) subgraph. When $m=o(n)$, one may not use a Tur\'an
graph as a construction, and a variety of interesting constructions and
methods were developed as a result. Ramsey--Tur\'an theory has been heavily
studied in the last half-century; see, e.g., the nice
survey by Simonovits and S\'os \cite{ss}. Our new general problem is
precisely the Ramsey--Tur\'an problem in the case where $F$ is an
independent set of order $m$.

Another question which has received much study is the Erd\H{o}s--Hajnal
Problem, which seeks to prove that if a graph $F$ is forbidden as an
induced subgraph, then there is always a large clique or a large
independent set.  The Erd\H{o}s--Hajnal Conjecture \cite{erdos-hajnal}
states that for any fixed $F$, there is a constant $c>0$ such that every
$F$-induced-free graph on $n$ vertices contains a clique or independent set
of order at least $n^c$, which is much larger than what is guaranteed
without the $F$-induced-free condition. This problem has been the focus of
extensive research (see, e.g., the survey of Chudnovsky \cite{chudnovsky}).
The relationship to our new problem is that an upper bound on
$\exi{n}{K_t}{F}$ of the form $nd/2$ implies an average degree of at most
$d$. Tur\'an's theorem then guarantees an independent set of order at least
$\frac{n}{d+1}$. This shows that a graph with no induced copy of $F$ contains either a clique of size $t$ or an independent set of size $\frac{n}{d+1}$. We will discuss this further in the concluding remarks.

\subsection{New results}

Throughout this paper, we consider only non-complete graphs $F$. Our new
function $\exi{n}{H}{F}$ sometimes reduces to the ordinary Tur\'an number
$\ex(n, \{H, F\})$ where both $H$ and $F$ are forbidden as (not necessarily
induced) subgraphs. Indeed, if $H = C_3$ and $F = C_4$, every graph which
is both $C_3$-free and $C_4$-induced-free is also $C_4$-free, and every graph
which is $C_4$-free is obviously $C_4$-induced-free.

As mentioned early in the introduction, if $F$ is non-bipartite, the Erd\H{o}s-Stone-Simonovits
theorem establishes that $\ex(n, F\ind)$ and $\ex(n, F)$ are both quadratic
in $n$. However, for bipartite $F$, $\ex(n, F\ind) = \binom{n}{2}$, while
the K\H{o}v\'ari--S\'os--Tur\'an theorem trivially establishes a sub-quadratic
upper bound $n^{2-\frac{1}{s}}$ for some $s$ for which $F \subset K_{s,t}$.
The two functions therefore deviate asymptotically for all bipartite $F$.
Our first main result shows that in fact, when $F = K_{s,t}$, we can
recover the same asymptotic upper bound as K\H{o}v\'ari--S\'os--Tur\'an by
forbidding \emph{any}\/ other fixed graph $H$.

\begin{theorem}\label{main theorem}
  If $G$ is an $n$-vertex graph with no copy of $K_r$ as a subgraph and no
  copy of $K_{s,t}$ as an induced subgraph, then 
  \begin{displaymath}
    e(G) 
    \leq 
    n^{2-1/s}4^s\left((r+t)^{t/s} + (r+s) + 2(r+t)^{t(s+1)/s^2}(r+s)\right) 
    + 2 \cdot 4^s n.
  \end{displaymath}
\end{theorem}

As a corollary, this shows that for any positive integers $s$ and $t$ and
any fixed graph $H$, 
\[\exi{n}{H}{K_{s,t}} = O\left(n^{2-\frac{1}{s}}\right)\]
where the implied constant depends on $H$, $s$, and $t$.  Note that if the
forbidden induced subgraph $F$ is bipartite but not complete bipartite,
then the complete bipartite graph $K_{n/2, n/2}$ provides a construction
which shows that $\exi{n}{K_r}{F}$ is quadratic in $n$ for all $r > 2$.

We will give a short proof of a slightly weaker version of Theorem \ref{main theorem} using dependent random choice. We then prove the full statement using a method that draws another connection between this problem and a recent
Tur\'an-type problem of Alon and Shikhelman \cite{as}. For graphs $T$ and
$H$, denote by $\ex(n, T, H)$ the maximum number of copies of $T$ in an
$H$-free graph with $n$ vertices. When $T = K_2$, this is the classical
Tur\'an number.  Several authors have studied this problem before (cf.
\cite{bg, erdos, hhknr}), and \cite{as} is the first systematic study of
the parameter. A key ingredient in the proof of our main theorem gives an
upper bound on the number of complete subgraphs in a graph that does not
contain $H$ or an induced copy of $K_{s,t}$. In particular, in \cite{as},
the quantity $\ex(n, K_m, K_{s,t})$ is studied. The following theorem is a
natural extension of $\ex(n, K_m, K_{s,t})$ to graphs with no
\emph{induced}\/ copy of $K_{s,t}$. We used it as a tool for proving
Theorem \ref{main theorem}, but due to the connection with Alon and
Shikhelman's problem, it may be of independent interest.

\begin{theorem}\label{counting cliques}
Let $G$ be an $n$-vertex, $K_r$-free graph with no copy of $K_{s,t}$ as an induced subgraph. 
If $t_m(G)$ is the number of cliques of order $m$ in $G$, then
\[
m\cdot t_m(G) \leq 2(t+r)^{tm/s}(r+s)^sn^{m-\frac{m-1}{s}} + (r+s)^sn^{m-1} .
\]
\end{theorem}

\subsection{Sharper results for special families}

In this section, we dive deeper into the constant factors, opening the
study with specific families of graphs for $F$ and $H$ in $\exi{n}{H}{F}$.
As was historically studied by others in graphs, we start with complete
bipartite graphs and cycles.  Theorem \ref{main theorem} gives that if $G$
is a graph with no induced copy of $K_{s,t}$ and $G$ has significantly more
than $n^{2-1/s}$ edges, then $G$ must contain a large complete subgraph.
This leads us to the work of Gy\'{a}rf\'{a}s, Hubenko, and Solymosi on
cliques in graphs with no induced $K_{2,2}$.  In \cite{ghs}, answering a
question of Erd\H{o}s, they show that any $n$-vertex graph with no induced
$K_{2,2}$ must have a clique of order at least $\frac{d^2}{10n}$, where $d$
is the average degree.  We extend this result to graphs with no induced
$K_{2,t}$.  In this special case, we obtain a much better bound than what
is implied by Theorem \ref{main theorem}. Here and in the remainder, the
clique number $\omega(G)$ denotes the maximum order of a clique contained
in $G$.

\begin{theorem}\label{cliques and k2t}
Let $t \geq 2$ be an integer.  If $G$ is a graph with $n$ vertices, minimum
degree $d$, and no induced $K_{2,t+1}$, then
\[
\omega (G)  \geq \left( \frac{d^2}{2nt} \left( 1 - o(1)  \right) \right)^{1/t} - t.
\]
\end{theorem} 

\begin{corollary}\label{cor}
Let $H$ be a graph with $v_H$ vertices.  For any integer $t \geq 2$, 
\[
  \exi{n}{H}{K_{2,t+1}} < ( \sqrt{2} + o(1) ) t^{1/2} ( v_{H} +t)^{t/2 } n^{3/2}.
\]
\end{corollary}
\begin{proof}
Let $t \geq 2$ be an integer and let $H$ be a graph with $v_H$ vertices.
Suppose that $G$ is an $n$-vertex $H$-free graph with no induced $K_{2,t+1}$.
Let $d$ be the average degree of $G$.  Let $G'$ be an $H$-free subgraph of
$G$ with minimum degree $d /2$ and no induced $K_{2,t+1}$.  By Theorem \ref{cliques and k2t}, $G'$
has a clique with at least $(1 - o(1)) \left( \frac{ d^2}{8nt} \right)^{1/t} - t$
vertices.  Since $G'$ is $H$-free, $G'$ cannot have a clique of order $v_H$ so 
\[
 (1 - o (1) ) \left( \frac{d^2}{8nt } \right)^{1/t} -t < v_H.
\]
Since $d = 2 e(G) / n$, we can solve this inequality for $e(G)$ to get 
\[
e(G) < ( \sqrt{2} + o (1)) t^{1/2}( v_{H} +t)^{t/2} n^{3/2}.
\]
\end{proof}

When $\chi (H) \geq 3$, we can obtain a lower bound of the same order of
magnitude by considering a max cut in a $K_{2,t+1}$-free graph with $n$
vertices and $\frac{1}{2}\sqrt{t}n^{3/2} - o(n^{3/2} )$ edges.  Such graphs
were constructed by F\"{u}redi in \cite{fur}.  A max cut in a
$K_{2,t+1}$-free graph will clearly not contain a copy of $H$ and will not
contain an induced copy of $K_{2,t+1}$.  This gives a lower bound of 
\begin{equation}\label{simple lower bound}
\frac{1}{4} \sqrt{t} n^{3/2} - o( n^{3/2} ) 
\leq 
\exi{n}{H}{K_{2,t+1}}
\end{equation}
for any $t \geq 2$ and non-bipartite $H$.

Theorem \ref{main theorem} shows that when one forbids induced copies of
$K_{s,t}$ and any other subgraph, the number of edges is bounded above by
something that is the same order of magnitude than what is given by the
K\H{o}v\'ari--S\'os--Tur\'an theorem, leaving the question of the
multiplicative constant. We have also remarked that there are instances
where the problem reduces to the ordinary Tur\'an number, for example
$\exi{n}{C_3}{C_4} = \ex(n, \{C_3, C_4\})$. 
A nice construction based on the incidence graph of a projective plane was used by
Bollob\'{a}s and Gy\"{o}ri \cite{bg} to show that there are $C_5$-free $n$-vertex graphs with 
many triangles.  It turns out that this same construction shows that for any $q$ that is a power 
of a prime, 
\[
  \exi{3(q^2 + q + 1)}{C_5}{C_4} \geq 2(q + 2)(q^2 + q + 1).
\]
A standard densities of primes argument then gives
\[
  \exi{n}{C_5}{C_4} \geq \frac{2}{ 3 \sqrt{3} } n^{3/2 } - o(n^{3/2} ),
\]
while Erd\H{o}s and Simonovits \cite{es} proved that 
\[
\ex( n , \{ C_4 , C_5 \} ) \leq \frac{1}{ 2 \sqrt{2} } n^{3/2} + 4 \left(
\frac{n}{2} \right)^{1/2}.
\]
This shows that there are situations when the numbers $\exi{n}{H}{F}$ and $\ex(n,
\{H, F\})$ may have different multiplicative constants.

Finally, we note that while Theorem \ref{main theorem} gives an upper bound
matching the K\H{o}v\'ari--S\'os--Tur\'an theorem in order of magnitude, the
multiplicative constant is dependent on certain Ramsey numbers in $r$, $s$,
and $t$, and so is likely not tight. Our final results display how one may
lower the multiplicative constant when one knows more about the forbidden
(not necessarily induced) subgraph $H$. We state the following theorem for
$H$ an odd cycle, but emphasize that the proof technique could be applied to a
wide family of graphs.

%%%%%%%%%%%

\begin{theorem}\label{odd cycles}
For any integers $k \geq 2$ and $t \geq 2$, there is a constant 
$\beta_k$, depending only on $k$, such that 
\[
  \exi{n}{C_{2k+1}}{K_{2,t}}
\leq 
( \alpha ( k , t )^{1/2} + 1 )^{1/2} \frac{ n^{3/2} }{2} + \beta_k n^{1 + 1/2k}
\]
where $\alpha (k,t) = ( 2k -2)(t - 1)( ( 2k -2)(t - 1) - 1)$.  
\end{theorem}

Observe that \eqref{simple lower bound} gives a lower bound on
$\exi{n}{C_{2k+1}}{K_{2,t}}$ since $C_{2k+1}$ is not bipartite. Therefore, Theorem \ref{odd cycles} is sharp in both order of magnitude and its dependence on $t$. We leave open the question of whether Theorem \ref{odd cycles} gives the correct growth rate as a function of $k$.

%%%%%%%%%%%%%%%%%%%%%%%%%%%%%%%%%%%%%%%%%%%%%%%%

\subsection{Notation and organization}

Let the Ramsey number $R(s,t)$ denote the smallest $n$ such that in any red
and blue coloring of the edges of $K_n$, there is either a red $K_s$ or a
blue $K_t$.  We write $t_m(G)$ for the number of complete subgraphs of $G$
that have exactly $m$ vertices.  An independent set of order $s$ is called
an \emph{$s$-independent set}. We define
\[
\mathcal{I}_s (G) = \{ \{x_1 ,  \dots , x_s \} \subset V(G) : x_1 ,  \dots
, x_s ~ \mbox{are distinct and non-adjacent in}~G \}.
\]
Similarly, a clique of order $m$ is called an \emph{$m$-clique}, and
$\mathcal{K}_m (G)$ denotes the set of all $m$-cliques in $G$.  Given a set
of vertices $\{x_1 , \dots , x_s \} \subset V(G)$, we write $N(x_1 , \dots
, x_s )$ to denote the vertices in $G$ that are adjacent to every vertex in
the set $\{x_1 , \dots , x_s \}$, and we let 
\[
d(x_1 , \dots , x_s ) = | N (x_1 , \dots , x_s ) |.
\]
We write $\Gamma ( x_1 , \dots , x_s )$ for the subgraph of $G$ induced by
the vertices in $N(x_1 , \dots , x_ s)$.  Lastly,  $\overline{H}$ denotes
the complement of the graph $H$.

This paper is organized as follows. We prove our two main results,
Theorems \ref{main theorem} and \ref{counting cliques}, in Sections
\ref{main theorem section} and \ref{counting cliques section}. Theorem
\ref{cliques and k2t} is proved in Section \ref{cliques and k2t}. We 
%give the construction of Theorem \ref{construction} in Section \ref{construction section} and 
prove Theorem \ref{odd cycles} 
%and \ref{kss} 
in Subsection \ref{H a cycle section}. The final section contains some concluding remarks
and open problems.

\section{The number of edges in $H$-free graphs with no induced $K_{s,t}$}\label{main theorem section}

In this section, let $r,s$, and $t$ be fixed positive integers.  Let $G$ be
an $n$-vertex graph with no copy of $K_r$ as a subgraph and no copy of
$K_{s,t}$ as an induced subgraph. We will prove Theorem \ref{main theorem},
showing that $e(G) = O\left(n^{2-1/s}\right)$, where the implied constant
depends on $r,s$, and $t$. First we give a short proof using dependent random choice, which gives a slightly worse constant than in Theorem \ref{main theorem}.

\begin{lemma}[Dependent Random Choice Lemma \cite{drc}]\label{drc lemma}
Let $a, r, s$ be positive integers and 
$G$ be a graph with $n$ vertices and average degree $d$.
If there is a positive integer $t$ such that
\begin{equation}\label{drc ineq}
\frac{d^t}{n^{t-1}} - \binom{n}{s} \left(\frac{r}{n}\right)^t \geq a,
\end{equation}
then $G$ contains a subset $A$ of at least $a$ vertices such that every set of 
$s$ vertices in $A$ has at least $r$ common neighbors.
\end{lemma}

\begin{proof}
Let $G$ be an $n$-vertex graph with no copy of $K_r$ and 
$e(G) \geq c n^{2 - 1/s}$.  We must show 
that $G$ contains an induced $K_{s,t}$. 
Writing $d$ for the average degree of $G$, we have 
\begin{align*}
\frac{d^s}{n^{s-1}} - \binom{n}{s} \left(\frac{R(r,t)}{n}\right)^s
& \geq
\frac{(2cn^{1-1/s})^s}{n^{s-1}} - \left(\frac{en}{s}\right)^s \left(\frac{R(r,t)}{n}\right)^s \\
& =
(2c)^s - \left(\frac{e R(r,t)}{s}\right)^s \\
& \geq
R(r,s).
\end{align*}
The last inequality holds provided $c$ is large as a function of $r$, $s$, and $t$.  
We conclude that there is a set of at least $R(r,s)$ vertices, say $A$, such that 
every set of $s$ vertices in $A$ have at least $R(r,t)$ common neighbors.  
Since $G$ has no $K_r$, $G[A]$ contains an independent
set $S$ of size $s$.  Furthermore, the vertices in $S$ have at least $R(r,t)$ common neighbors.
Again, since $G$ has no $K_r$, $G[N(S)]$ contains an independent
set $T$ of size $t$.
Therefore, $G[S \cup T]$ is an induced copy of $K_{s,t}$.
\end{proof}

Now we give a full proof of Theorem \ref{main theorem}. The proof will rely on an upper bound on the
number of cliques of a fixed order in $G$, for which we will apply Theorem
\ref{counting cliques}. We will delay the proof of Theorem \ref{counting
cliques} to Section \ref{counting cliques section}. We will need the
following claim which also counts cliques. A much stronger version is given
by Conlon in \cite{conlon}, but we only need a weaker version that can be
proved using an elementary counting argument of Erd\H{o}s \cite{erdos}.
\medskip
\begin{lemma}\label{ramsey multiplicity}
If $F$ is a graph on $n > 2\cdot 4^s$ vertices, then 
\[
t_s(F) + t_s(\overline{F}) \geq \frac{n^s}{2^s 4^{s^2}}.
\]
\end{lemma}
\begin{proof}
Since it is well known that $R(s,s) < 4^s$, any set of
$4^s$ vertices in $V(F)$ must contain a clique of order $s$ in either $F$
or $\overline{F}$.  Each set of $s$ vertices is contained in
$\binom{n-s}{4^s-s}$ sets of order $4^s$. Therefore,
\[
t_s(F) + t_s(\overline{F}) \geq \frac{\binom{n}{4^s}}{\binom{n-s}{4^s-s}} > \frac{\left(n-4^s\right)^s}{\left(4^s\right)^s} > \frac{n^s}{2^s4^{s^2}}
\]
where in the last inequality we have used the assumption that $n > 2 \cdot 4^s$.  
\end{proof}

\medskip

\begin{proof}[Proof of Theorem \ref{main theorem}]
Let $G$ be an $n$-vertex graph that is $K_r$-free and has no induced copy of $K_{s,t}$.  
We must show that $e(G) < c n^{2  - 1/s}$ where $c$ is a constant depending only on $r$, $s$, and $t$.  
By repeatedly removing vertices of degree less than $2\cdot 4^s$, we may assume that $G$ has minimum degree at least $2\cdot 4^s$ without loss of generality. In particular, 
when the minimum degree is at least $2 \cdot 4^s$ we can apply Lemma \ref{ramsey multiplicity} 
to the graph $\Gamma (v)$ for any $v \in V(G)$.  We now proceed with the main part of the proof of Theorem \ref{main theorem}.      

Since $G$ does not contain an induced copy of $K_{s,t}$, an $s$-independent set cannot contain a $t$-independent set 
in its common neighborhood.  Also, no set of vertices can contain a clique
of order $r$ in its neighborhood since $G$ is $K_r$-free.
We conclude that for any $\{x_1,\dots, x_s\} \in \mathcal{I}_s(G)$,
\[
d(x_1,\dots, x_s) \leq R(r,t). 
\]
Therefore, using the Erd\H{o}s--Szekeres \cite{erdos-szekeres} bound $R(r,t)
\leq \binom{r+t-2}{t-1}$,
\begin{equation}\label{upper bound over I}
\sum_{\{x_1, \dots , x_s\} \in \mathcal{I}_s(G)} d(x_1,\dots, x_s) \leq \binom{n}{s} R(r,t) < (r+t)^t n^s.
\end{equation}
On the other hand, we may double count to see that
\[
\sum_{\{x_1,\dots, x_s\}\in \mathcal{I}_s(G)} d(x_1,\dots x_s) = \sum_{v\in V(G)} t_s\left(\overline{\Gamma(v)}\right).
\]
Using Lemma \ref{ramsey multiplicity} and then convexity, we get
\begin{align*}
\sum_{ \{x_1 , \dots , x_s \}  \in \mathcal{I}_s(G) } d(x_1 , \dots , x_s) 
&\geq
\sum_{v \in V(G) } \left( \frac{ d(v)^s }{ 2^s 4^{s^2} } - t_s ( \Gamma (v) ) \right) \\
&\geq \frac{n}{2^s 4^{s^2} }\left(  \frac{1}{n} \sum_{v \in V(G) } d(v) \right)^s - \sum_{v \in V(G) } t_s ( \Gamma (v) ) 
\\
&=
\frac{n}{2^s 4^{s^2} } \left( \frac{ 2 e(G) }{n} \right)^s - (s + 1) t_{s+1} (G) 
\\
&=
\frac{ ( e(G)) ^s }{ n^{s-1} 4^{s^2} } - (s + 1) t_{s+1}(G).
\end{align*}
This inequality, together with (\ref{upper bound over I}), gives
\begin{equation*}
( r + t)^t n^s \geq \frac{ ( e(G)) ^s }{ n^{s-1} 4^{s^2} } - (s + 1) t_{s+1}(G).
\end{equation*}
By Theorem \ref{counting cliques}, 
\[
(s+1) \cdot t_{s + 1} (G) \leq 2 ( t + r)^{ t(s+1) / s} ( r + s)^s n^s + (r + s)^s n^s
\]
and this estimate, together with the previous inequality, gives the result.   
\end{proof}

%A key ingredient in the proof just given is Theorem \ref{counting cliques}.  
%In fact, Theorem \ref{counting cliques} was discovered during our attempts to prove 
%Theorem \ref{main theorem}.  There is a simpler proof that uses the dependent 
%random choice method, which is excellently surveyed in the paper 
%of Fox and Sudakov \cite{drc}, and has numerous applications to problems in 
%extremal combinatorics.  Here we will use this method to give alternative proof to 
%Theorem \ref{main theorem}.  

%%%%%%%%%%%%%%%%%%%%%%%%%%%%%%%%%%%%%%%%%%%

\section{Clique counting with forbidden induced subgraphs}\label{counting cliques section}

As in the previous section, $r$, $s$, and $t$ are positive integers.  In
this section we prove our upper bound on the number of $m$-cliques in any
$n$-vertex, $K_r$-free graph with no induced copy of $K_{s,t}$.

Let $G$ be an $n$-vertex graph that is $K_r$-free and has no $K_{s,t}$ as an induced subgraph.  
We will write $\mathcal{I}_s$ for $\mathcal{I}_s
(G)$ and $\mathcal{K}_{m-1}$ for $\mathcal{K}_{m-1}(G)$.  Consider the set
of pairs 
\[
S:= \left\{ ( \{ x_1,\dots, x_s \}, v) : \{x_1,\dots, x_s\} \in \mathcal{I}_s, v\in \Gamma(x_1, \dots, x_s)\right\}.
\]
As observed in the proof of Theorem \ref{main theorem}, 
the common neighborhood of an $s$-independent set has no $t$-independent set or an $r$-clique. 
Therefore,
\begin{equation}\label{S upper bound}
|S| = \sum_{ \{ x_1,\dots, x_s \} \in \mathcal{I}_s} d(x_1,\dots, x_s) \leq \binom{n}{s} R(t, r) < (t+r)^t n^s.
\end{equation}

To give a lower bound on $|S|$, we count from the perspective of $(m-1)$-cliques with $s$-independent sets in their neighborhood.

\begin{lemma}\label{counting s sets}
If $\{x_1,\dots, x_{m-1}\} \in \mathcal{K}_{m-1}$ and $d(x_1,\dots, x_{m-1}) > R(r-m+1, s)$, then the number of $s$-independent sets in $\Gamma(x_1,\dots, x_{m-1})$ is at least 
\[
\left(\frac{d(x_1,\dots, x_{m-1})}{2(r+s)^s}\right)^s
\]
\end{lemma}
\begin{proof}
If $\{ x_1,\dots, x_{m-1} \}$ forms a clique in $G$, then its neighborhood
can have no clique of order $r-m+1$. Thus, every set of order $R(r-m+1,s)$
in its neighborhood must contain an $s$-independent set. Now, any set of
order $s$ in $\Gamma(x_1,\dots, x_{m-1})$ is contained in at most
$\binom{d(x_1,\dots, x_{m-1})-s}{R(r-m+1, s)-s}$ sets of order $R(r-m+1, s)$. Therefore, 
\[
| \mathcal{I}_s ( \Gamma( x_1 , \dots , x_{m-1} ) ) |  \geq \frac{\binom{d(x_1,\dots, x_{m-1})}{R(r-m+1,s )}}{\binom{d(x_1,\dots, x_{m-1}) -s}{R(r-m+1,s)-s}} \geq \frac{(d(x_1,\dots, x_{m-1})-s)^s}{R(r-m+1, s)^s}.
\]

Using the estimates $d(x_1 , \dots , x_{m-1}) - s > d(x_1,\dots , x_{m-1})/2$ and $R(r-m+1,s) < (s+r)^s$ proves the claim.
\end{proof}

\bigskip

With Lemma \ref{counting s sets}, we are now ready to prove Theorem \ref{counting cliques}

\begin{proof}[Proof of Theorem \ref{counting cliques}]
Let $G$ be an $n$-vertex graph with no $K_r$ and no induced $K_{s,t}$.  
The vertices in an $s$-independent set have at most $R(t,r)$ common neighbors.  Thus,
each $s$-independent set may be contained in the common neighborhood of at
most 
\[
\binom{R(t,r)}{m-1} < (r+t)^{t(m-1)}
\]
$(m-1)$-cliques.   
Let $B \subset \mathcal{K}_{m-1}$ be the $(m-1)$-cliques in $G$ where 
the vertices of each $(m-1)$-clique in $B$ have more than $R(r-m+1,s)$ common neighbors. We have 
\begin{align*}
|S| & \geq \frac{1}{(r+t)^{t(m-1)}} \sum_{\{x_1,\dots, x_{m-1}\} \in B} 
| \mathcal{I}_s ( \Gamma (x_1 , \dots , x_{m-1} ) ) | \\
&\geq \frac{1}{(r+t)^{t(m-1)}} \sum_{\{x_1,\dots, x_{m-1}\}\in B} \left(\frac{d(x_1,\dots, x_{m-1})}{2(r+s)^s}\right)^s   \\
& \geq \frac{1}{(r+t)^{t(m-1)}} \frac{\left(\sum_B d(x_1,\dots, x_{m-1})\right)^s}{2^s(r+s)^{s^2}|B|^{s-1}}  
\end{align*}
where the second inequality is by Lemma \ref{counting s sets} and the third
is by convexity.  Since 
\[
mt_m(G) =\sum_{\{x_1,\dots, x_{m-1}\}\in \mathcal{K}_{m-1}} d(x_1,\dots , x_{m-1}),
\]
we have
\begin{align*}
\sum_{\{x_1,\dots, x_{m-1} \}\in B} d(x_1,\dots, x_{m-1}) &\geq mt_m(G) - \binom{n}{m-1} R( r-m+1,s)
 \\ &\geq mt_m(G) - (r+s)^s n^{m-1} .
\end{align*}
Combining this inequality with our lower bound on $|S|$ and the trivial inequality $|B| < n^{m-1}$ gives  
\begin{equation}\label{S lower bound}
|S| \geq \frac{ \left(mt_m(G) - (r+s)^sn^{m-1}\right)^s}{(r+t)^{t(m-1)}2^s(r+s)^{s^2}n^{(s-1)(m-1)}}.
\end{equation}
Combining \eqref{S upper bound} and \eqref{S lower bound} finishes the
proof of Theorem \ref{counting cliques}.
\end{proof}

%%%%%%%%%%%%%%%%%%%%%%%%%%%%%%%%%%%%%%

\section{Sharper results for $\boldsymbol{K_{2,t+1}}$}\label{cliques and k2t section}

In this section we prove Theorem \ref{cliques and k2t} and Corollary
\ref{cor}.  We must show that a graph with $n$ vertices, minimum degree
$d$, and no induced $K_{2,t+1}$ must have a clique of order at least $(1 - o
(1)) \left( \frac{d^2}{2nt} \right)^{1/t}-t$.  Our argument extends the
methods of Gy\'{a}rf\'{a}s, Hubenko, and Solymosi \cite{ghs}.  

\bigskip

\begin{proof}[Proof of Theorem \ref{cliques and k2t}]
Let $G$ be a graph with $n$ vertices, minimum degree $d$, and no induced 
copy of $K_{2,t+1}$.  Let $\alpha = \alpha (G)$ and let 
$S$ be an independent set of size $\alpha$, say $S = \{x_1 , \dots , x_{ \alpha } \}$. 
Let $B_i$ be the vertices in $G$ whose only neighbor in $S$ is $x_i$.  
Let $B_{i,j}$ be the vertices in $G$ adjacent to both $x_i$ and $x_j$ (and possibly 
other vertices of $S$).  Since $S$ is an independent set with the maximum 
number of vertices, each $B_i$ is a clique and 
so $\{x_i \} \cup B_i$ is a clique.  
Also, 
\begin{equation}\label{union}
V(G) = \left( \bigcup_{i = 1}^{ \alpha } ( \{x_i \} \cup B_i )    \right) \bigcup 
\left( \bigcup_{1 \leq i < j \leq \alpha } B_{i,j} \right)
\end{equation}
otherwise we could create a larger independent set by adding a vertex to $S$.  
If $| \{x_i \} \cup B_i | \geq \left( \frac{ d^2}{2nt} \right)^{1/t}$ for some 
$i \in \{1,2, \dots , \alpha \}$, then we are done as 
$\{ x_i \} \cup B$ is a clique.  Assume that this is not the case.
By (\ref{union}), 
\[
n \leq \alpha \left( \frac{d^2}{ 2nt } \right)^{1/t} + \sum_{1 \leq i < j \leq \alpha} |B_{i,j} | .
\]
By averaging, there is a pair $1 \leq i < j \leq \alpha$ such that  
\[
|B_{i,j} | \geq \frac{ n - \alpha \left( \frac{d^2}{ 2nt } \right)^{1/t} }{ \binom{\alpha }{2} }.
\]
The set $B_{i,j}$ cannot contain a $(t+1)$-independent set otherwise we have an induced $K_{2,t+1}$ 
using the vertices $x_i$ and $x_j$.  If $w$ is any integer for which 
\begin{equation}\label{revised eq1}
\frac{ n - \alpha \left( \frac{ d^2 }{ 2nt } \right)^{1/t} }{ \binom{ \alpha }{2} } \geq R( t + 1 , w ),
\end{equation}
then $B_{i,j}$ contains a clique with $w$ vertices.  

If $\alpha (G) < \frac{2n}{d}$, then a short calculation gives  
\[
 \frac{ n - \alpha \left( \frac{d^2}{ 2nt } \right)^{1/t} }{ \binom{\alpha }{2} }
 \geq 
 \frac{ d^2}{2n} \left( 1 - \frac{2}{d} \left( \frac{d^2}{2nt } \right)^{1/t} \right)
  \geq
 \frac{d^2}{2n} \left( 1 - \frac{2}{ n^{1/t} } \right).
 \]
The second inequality holds since if $t \geq 2$, then 
$\frac{2}{d} \left( \frac{d^2}{2n} \right)^{1/t} \leq \frac{2}{ n^{1/t} }$. 
By the Erd\H{o}s--Szekeres bound on Ramsey numbers, $R( t + 1, w) < ( t + w -1)^t$ so that if $w$ is an integer for 
which 
\[
 \frac{d^2}{2n} \left( 1 - \frac{2}{ n^{1/t} } \right)  \geq ( w + t-1)^t,
 \]
then $B_{i,j}$ contains a clique of size $w$.  We conclude that in the case when $\alpha (G) < \frac{2n}{d}$, we have 
\[
\omega (G) \geq \left( \frac{ d^2}{2n} \left( 1 - \frac{2}{ n^{1/t} } \right) \right)^{1/t } - t.
\]
 
Now assume that $\alpha (G) \geq \frac{2n}{d}$.  Let $b = \frac{2n}{d}$ and 
let $\{ x_1 , \dots , x_b \}$ be an independent set.  
If $m = \max_{i \neq j } | N(x_i ,x_j) |$, then
\[
b d - \binom{b}{2} m \leq \sum_{i=1}^{b} | N(x_i) | - \sum_{1 \leq i < j \leq b } | N(x_i ,x_j) | \leq \left| \bigcup_{i=1}^{b} N(x_i) \right| \leq n
\]
which implies 
\[
m \geq \frac{ bd - n }{ \binom{b}{2} } .
\]
Fix a pair $1 \leq i < j \leq b$ with $| N(x_i ,x_j) | = m$.  If  
$N(x_i ,x_j)$ contains an independent set of order $t+1$, then we get an induced $K_{2,t+1}$.  
As before, if $w$ is any integer for which 
\[
\frac{ bd -n }{ \binom{b}{2} } \geq R( t + 1 , w),
\]
then $N(x_i, x_j)$ contains a clique with $w$ vertices.  Since $b = \frac{2n}{d}$, 
we have 
\[
\frac{bd - n }{ \binom{b}{2} } \geq \frac{d^2}{2n},
\]
and so $\frac{d^2}{2n} \geq (w + t - 1 )^t$ implies that $\omega (G) \geq w$.  We conclude that in the case when 
$\alpha (G) \geq \frac{2n}{d}$, 
\[
\omega (G) \geq \left( \frac{d^2}{2n} \right)^{1/t} - t.
\]
This completes the proof of Theorem \ref{cliques and k2t}.
\end{proof}

%%%%%%%%%%%%%%%%%%%%%%%%%%%%%%%%%%%%%%%%%%%%%%%%%%%%%%%%%

\subsection{Forbidding an odd cycle}\label{H a cycle section}

In this section we prove Theorem \ref{odd cycles}.  We must show that for integers $k \geq 2$ and $t \geq 2$, 
any $n$-vertex $C_{2k+1}$-free graph with no induced $K_{2,t}$  has at most 
\[
\left(  \alpha( k , t)^{1/2} + 1  \right)^{1/2}  \frac{ n^{3/2} }{2}  + \beta_k n^{1 + 1/2k}
\]
edges where $\alpha(k,t) = ( 2k - 2)(t - 1) ( ( 2k - 2)(t - 1) - 1)$.  

\bigskip

\begin{proof}[Proof of Theorem \ref{odd cycles}]
Suppose $G$ is a $C_{2k+1}$-free graph with $n$ vertices and no induced copy of $K_{2,t}$.  
For any pair of distinct non-adjacent vertices $x$ and $y$, the common neighborhood $N(x, y ) $ cannot 
contain a path of length $2k - 1$ or an independent set of order $t$.  A classical result of Erd\H{o}s and Gallai is that 
any graph with at least $(a - 1)(b - 1) + 1$ vertices must contain a path
of length $a$ or an independent set of order $b$
(see Parsons \cite{parsons}).  
Therefore, 
\begin{equation}\label{eq 1.1}
d( x,y)  \leq ( 2k-2)(t - 1).
\end{equation}
Let $\overline{e} = \binom{n}{2} - e(G)$.  By convexity and 
(\ref{eq 1.1}), 
\begin{equation}\label{eq 1.2}
\frac{ \alpha( k , t) }{2} \binom{n}{2} \geq \sum_{ \{x , y \} \notin E(G) } 
\binom{ d(x,y) }{2} 
\geq 
\overline{e} 
\binom{
\frac{1}{ \overline{e} } \sum_{ \{x , y \} \notin E(G) } d(x,y) }{2}
\end{equation}
where $\alpha (k,t) := ( 2k - 2) ( t - 1)( (2k-2)(t-1) - 1)$.  
Note that 
\begin{equation}\label{eq 1.25}
\sum_{ \{x,y \} \notin E(G) } d(x,y) = \sum_{z \in V(G) } \left( \binom{d(z)}{2} - e ( \Gamma (z) ) \right) = 
\sum_{z \in V(G) } \binom{ d(z) }{2} - 3 t_3 (G).
\end{equation}
By convexity, 
\begin{equation}\label{eq 1.3}
\sum_{z \in V(G) } \binom{d(z)}{2} \geq n \binom{ 2e / n }{2}
\end{equation}
where $e$ is the number of edges of $G$.  By a result of Gy\"{o}ri and Li \cite{gl}, since 
$G$ is $C_{2k+1}$-free the number of triangles in $G$ is at most
$( c_k / 3)  n^{1 + 1/k} $.  
Here $c_k$ is a constant depending only on $k$.  
This fact, together with  
(\ref{eq 1.25}) and (\ref{eq 1.3}), give
\[
\sum_{ \{ x, y \} \notin E(G) } d(x,y) \geq n \binom{ 2e / n }{2} - c_k n^{1 + 1/k} .
\]
Combining this with (\ref{eq 1.2}) leads to 
\begin{align*}
\frac{ \alpha ( k , t) }{2} \binom{n}{2} 
&\geq
\overline{e} 
\binom{
\frac{1}{ \overline{e} } \left( 
n \binom{2e / n }{2} - c_k n^{1 + 1/k}  \right) }{2} \\
 &\geq 
\frac{ \overline{e} }{2} \left(  \frac{ n}{ \overline{e}}  \binom{ 2e / n }{2} - \frac{c_k n^{1 + 1/k} }{ \overline{e} }- 1 \right)^2 \\
&= \frac{1}{2 \overline{e} } \left( n \binom{2 e / n }{2} - c_k n^{1 + 1/k} - \overline{e} \right)^2.
\end{align*}
Using the trivial estimate $\overline{e} \leq \binom{n}{2}$, we have 
\[
\alpha ( k , t) \binom{n}{2}^2 \geq  \left( n \binom{2 e / n }{2} - c_k n^{1 + 1/k} - \overline{e} \right)^2.
\]
A straightforward calculation gives
\[
\left( \alpha (k , t)^{1/2} + 1 \right) \binom{n}{2} + c_k n^{1 + 1/k} \geq \frac{2e^2}{n}
\]
from which it follows that 
\[
 \left( \alpha(k , t )^{1/2} + 1 \right)^{1/2} \frac{ n^{3/2} }{ 2 } + \sqrt{ \frac{ c_k }{2} } n^{1 + 1/2k} \geq e.
\]
\end{proof}

%%%%%%%%%%%%%%%%%%%%%%%%%%%%%%%%%%%%%%%%%%%%%%%%%%%%%%%%%

\section{Concluding remarks}

Our bound
in Theorem \ref{counting cliques} is probably not tight, and although it served our purposes
in this paper, it is an independently interesting question (along the lines
of Alon and Shikhelman's problem in \cite{as}) to resolve its asymptotic
behavior.  Apart from its natural interest, another side effect of an
improvement could also potentially translate into a constant factor
improvement in the Erd\H{o}s-Hajnal problem for forbidden $K_{s,t}$. (The
conjecture in this case has been known since the original paper of Erd\H{o}s and Hajnal \cite{erdos-hajnal}, which covered the more general
case of \emph{cographs}.) In connection with the Erd\H{o}s-Hajnal conjecture, we note the following corollary of Theorem \ref{main theorem}, which complements the work of Gy\'arf\'as, Hubenko, and Solymosi in \cite{ghs} and Theorem \ref{cliques and k2t}. One could be more careful estimating Ramsey numbers in our proofs to obtain a slightly improved exponent.

\begin{corollary}
If $G$ has average degree $d$ and no copy of $K_{s,t}$ as an induced subgraph, then
\[
\omega(G) = \Omega\left(\left(\frac{d^s}{n^{s-1}}\right)^{\frac{s}{t(s+1)+ s^2}}\right).
\]
\end{corollary}
{\scriptsize $\blacksquare$}

It also remains open to estimate $\exi{n}{H}{F}$ with greater accuracy. The
bounds would likely depend on the structures of $H$ and $F$. The results
from the later sections of our paper start this investigation by proving
some bounds in the case of odd cycles and $K_{2,t}$-induced.  It would be interesting if the behavior of
this function is sometimes determined by natural parameters of $H$ and $F$,
as in the case of the ordinary Tur\'an problem.

Finally, we note that using the same technique as in the proof of Theorem \ref{odd cycles}, the main result of \cite{bg}, and a result of Maclaurin now called the Fisher--Ryan inequalities \cite{fr}, one can show
\[
\exi{n}{C_{2k+1}}{K_{s,s}} \leq \frac{4^s(s-1)^{1/s}(2k-3)^{1/s}}{(s!)^{1/s}}n^{2-1/s} + o(n^{2-1/s}),
\]
where $k\geq s\geq 3$. Whereas Theorem \ref{odd cycles} has the correct dependence on $t$, we do not know if the above equation has the correct dependence on $s$.

%%%%%%%%%%%%%%%%%%%%%%%%%%%%%%%%%%%%%%%%%%%%%%%%%%%%%%%%%

%%%%%%%%%%%%%%%%%%%%%%%%%%%%%%%%%%%%%%%%%%%%%%%%%%%%%%%%%


\begin{thebibliography}{10}


\bibitem{as}
N.\ Alon and C.\ Shikhelman,
Many $T$ copies in $H$-free graphs,
{\emph J. Combin. Theory Ser. B} \textbf{121} (2016), 146--172.  

\bibitem{balogh-bollobas-weinreich}
  J. Balogh, B. Bollob\'as, and D. Weinreich, The speed of hereditary
  properties of graphs, \emph{J. Combin. Theory Ser. B} \textbf{79} (2000),
  131--156.

\bibitem{bermond-etal} J.\ Bermond, J.\ Bond, M.\ Paoli, and C.\ Peyrat,
  Graphs and interconnection networks: diameter and vulnerability, in:
  Surveys in Combinatorics: Proceedings of the Ninth British Combinatorics
  Conference, \emph{London Math. Soc., Lec. Notes Ser.} \textbf{82} (1983),
  1--30.

\bibitem{boehnlein-jiang-induced}
  E.\ Boehnlein and T.\ Jiang, Set families with a forbidden induced
  subposet, \emph{Combinatorics, Probability and Computing}
  \textbf{21} (2012), 496--511.

\bibitem{bg}
B.\ Bollob\'{a}s, E.\ Gy\"{o}ri, 
Pentagons vs.\ triangles, 
{\em Discrete Math}.\ 308 (2008), no.\ 19, 4332--4336.  

\bibitem{bollobas-thomason}
  B. Bollob\'as and A. Thomason, Hereditary and monotone properties of
  graphs, \emph{The Mathematics of Paul Erd\H{o}s, II} \textbf{14} (1997),
  70--78.

\bibitem{chudnovsky}
  M. Chudnovsky, The Erd\H{o}s--Hajnal conjecture---a survey,
  \emph{Journal of Graph Theory} \textbf{75} (2014), 178--190.

\bibitem{chung-gyarfas-tuza-trotter}
  F.R.K. Chung, A. Gy\'arf\'as, Z. Tuza, and W.T. Trotter, The maximum
  number of edges in $2K_2$-free graphs of bounded degree, \emph{Discrete
  Mathematics} \textbf{81} (1990), 129--135.

\bibitem{chung-jiang-west}
  M. Chung, T. Jiang, and D. West, Induced Tur\'an problems: largest
  $P_m$-free graphs with bounded degree, submitted.

\bibitem{chung-west-2p3}
  M. Chung and D. West, Large $2P_3$-free graphs with bounded degree,
  \emph{Discrete Mathematics} \textbf{150} (1996), 69--79.

\bibitem{chung-west-p4}
  M. Chung and D. West, Large $P_4$-free graphs with bounded degree,
  \emph{Journal of Graph Theory} \textbf{17} (1993), 109--116.

\bibitem{conlon}
D.\ Conlon,
On the Ramsey multiplicity of complete graphs,
{\em Combinatorica}, Vol. 32, Issue 2, (2012), 171--186.


\bibitem{erdos2}
P.\ Erd\H{o}s, 
On extremal problems of graphs and generalized graphs,
{\em Israel J.\ Math}.\ {\bf 2} 1964 183--190.  

\bibitem{erdos}
P.\ Erd\H{o}s, On the number of complete subgraphs contained in certain graphs, 
{\em Publ.\ Math.\ Inst.\ Hung.\ Acad.\ Sci., VII, Ser.\ A} {\bf 3} (1962), 459--464.  


\bibitem{efr-count}
  P.\ Erd\H{o}s, P.\ Frankl, and V.\ R\"odl, The asymptotic number of
  graphs not containing a fixed subgraph and a problem for hypergraphs
  having no exponent, \emph{Graphs Combin.} \textbf{2} (1986), 113--121.

\bibitem{erdos-hajnal}
  P.\ Erd\H{o}s and A. Hajnal, Ramsey-type theorems, \emph{Discrete Applied
  Mathematics} \textbf{25} (1989), 37--52.

\bibitem{es}
P.\ Erd\H{o}s and M.\ Simonovits,
Compactness results in extremal graph theory,
{\em Combinatorica} {\bf 2} (1982), no.\ 3, 275--288.  

\bibitem{erdos-stone}
  P.\ Erd\H{o}s and A.\ Stone, On the structure of linear graphs,
  \emph{Bulletin of the American Mathematical Society} \textbf{52} (1946),
  1087--1091.

\bibitem{erdos-szekeres}
  P.\ Erd\H{o}s and G.\ Szekeres, A combinatorial problem in geometry,
  \emph{Comp. Math.} \textbf{2} (1935), 463--470.

\bibitem{fr}
D.\ Fisher, J.\ Ryan,
Bounds on the number of complete subgraphs, 
{\em Discrete Math}.\ 103 (1992), no.\ 3, 313--320.  

\bibitem{drc}
J.\ Fox, and B.\ Sudakov, 
Dependent random choice
{\em Random Structures and Algorithms} 38 (2011), no.\ 1-2, 68--99.

\bibitem{fur}
Z.\ F\"{u}redi, 
New asymptotics for bipartite Tur\'{a}n numbers,
{\em J.\ Combin.\ Theory Ser.\ A} 75 (1996), no.\ 1, 141--144.

\bibitem{fs}
Z.\ F\"{u}redi and M.\ Simonovits, 
The history of degenerate (bipartite) extremal graph problems, 
{\em Erd\H{o}s centennial}, 169-264, Bolyai Soc.\ Math.\ Stud., 25, J\'{a}nos Bolyai Math.\ Soc., Budapest, 2013.


\bibitem{ghs}
A.\ Gy\'{a}rf\'{a}s, A.\ Hubenko, and J.\ Solymosi,
Large cliques in $C_4$-free graphs,
{\em Combinatorica} 22 (2002), no.\ 2, 269--274. 

\bibitem{gl}
E.\ Gy\"{o}ri and H.\ Li, 
The maximum number of triangles in $C_{2k + 1}$-free graphs, 
{\em Combin.\ Probab.\ Comput}.\ 21 (2012), no.\ 1-2, 187--191.  

\bibitem{hhknr}
H.\ Hatami, J.\ Hladk\'y, D.\ Kr\'al', S.\ Norine, and A.\ Razborov,
On the number of pentagons in triangle-free graphs,
{\em J. Combin. Theory Ser. A} 120 (2013), 722--732. 


\bibitem{lrz}
Y.\ Li, C.\ Rousseau, and W.\ Zang, 
Asymptotic upper boundes for Ramsey functions,
{\em Graphs Combin}.\ 17 (2001), no.\ 1, 123--128.  

\bibitem{lu-milans-induced}
  L.\ Lu and K. Milans, Set families with forbidden subposets,
  \emph{Journal of Combinatorial Theory, Series A} \textbf{136} (2015),
  126--142.

\bibitem{parsons}
T.\ D.\ Parsons, 
The Ramsey numbers $r( P_m , K_n)$,
{\em Discrete Math.} {\bf 6} (1973), 159--162.  


\bibitem{promel-steger1} H. Pr\"omel and A. Steger, Excluding induced
  subgraphs I: quadrilaterals, \emph{Random Structures and Algorithms}
  \textbf{2} (1991), 55--71.

\bibitem{promel-steger2} H. Pr\"omel and A. Steger, Excluding induced
  subgraphs II: extremal graphs, \emph{Discrete Applied Mathematics}
  \textbf{44} (1993), 283--294.

\bibitem{promel-steger3} H. Pr\"omel and A. Steger, Excluding induced
  subgraphs III: a general asymptotic, \emph{Random Structures and
  Algorithms} \textbf{3} (1992), 19--31.

\bibitem{razborov-induced}
  A.\ Razborov, On 3-hypergraphs with forbidden 4-vertex configurations,
  \emph{SIAM J. Discrete Math.} \textbf{24} (2010), 946--963.

\bibitem{sidorenko}
A.\ Sidorenko, 
What we know and what we do not know about Tur\'an numbers, 
{\em Graphs Combin.} 11 (1995), no.\ 2, 179--199.

\bibitem{ss}
M.\ Simonovits and V.\ T.\ S\'os,
Ramsey--Tur\'an theory.
Combinatorics, graph theory, algorithms and applications, 
{\em Discrete Math.}.\ 229 (2001), no.\ 1-3, 293--340.

\bibitem{sos}
V.\ T.\ S\'os,
On extremal problems in graph theory,
in Proceedings of the Calgary International Conference on Combinatorial Structures and their Application, 1969, 407--410.

%
%\bibitem{sos}
%V.\ T.\ S\'os,
%On extremal problems in graph theory,
%1970 {\em Combinatorial Structures and their Applications
%(Proc.\ Calgary Internat.\ Conf., Calgary, Alta., 1969)} pp.\ 407--410
%{\em Gordon and Breach, New York}
%in Proceedings of the Calgary International Conference on Combinatorial Structures and their Application, 1969, 407--410.


\bibitem{turan}
  P. Tur\'an, On an extremal problem in graph theory (in Hungarian),
  \emph{Mat. Fiz. Lapok} \textbf{48} (1941), 436--452.



\end{thebibliography}
\end{document}